\documentclass[a4paper]{amsart}
\usepackage{graphicx}
\usepackage{amssymb} 
\usepackage{stmaryrd}
\usepackage[mathcal]{euscript}
\usepackage{tikz}
\usepackage{tikz-cd}
\usepackage{hyperref}
\hypersetup{%
  bookmarksnumbered=true,%
  colorlinks=true,%
  linkcolor=blue,%
  citecolor=blue,%
  filecolor=blue,%
  menucolor=blue,%
  urlcolor=blue,%
  bookmarksopen=true,%
  bookmarksdepth=2,%
  pageanchor=true}

\title{The finitistic dimension of a triangulated category}

\author{Henning Krause}
\address{Fakult\"at f\"ur Mathematik\\
Universit\"at Bielefeld\\ D-33501 Bielefeld\\ Germany}
\email{hkrause@math.uni-bielefeld.de}

\theoremstyle{plain}
\newtheorem{thm}{Theorem}[section]
\newtheorem{prop}[thm]{Proposition}

\newtheorem{lem}[thm]{Lemma}

\theoremstyle{definition}
\newtheorem{defn}[thm]{Definition}
\newtheorem{exm}[thm]{Example}

\theoremstyle{remark}
\newtheorem{rem}[thm]{Remark}

\numberwithin{equation}{thm}

\newcommand{\add}{\operatorname{add}}

\newcommand{\amp}{\operatorname{amp}}

\newcommand{\cone}{\operatorname{cone}}

\renewcommand{\dim}{\operatorname{dim}}

\newcommand{\End}{\operatorname{End}}

\newcommand{\Ext}{\operatorname{Ext}}

\newcommand{\findim}{\operatorname{fin{.}dim}}
\newcommand{\Findim}{\operatorname{Fin{.}dim}}

\newcommand{\gldim}{\operatorname{gl{.}dim}}

\newcommand{\Hom}{\operatorname{Hom}}

\newcommand{\id}{\operatorname{id}}

\newcommand{\idim}{\operatorname{inj{.}dim}}

\newcommand{\Inj}{\operatorname{Inj}}

\newcommand{\Mod}{\operatorname{Mod}}

\newcommand{\pdim}{\operatorname{proj{.}dim}}
\newcommand{\Perf}{\operatorname{Perf}}

\newcommand{\rad}{\operatorname{rad}}

\newcommand{\RHom}{\operatorname{RHom}}

\newcommand{\thick}{\operatorname{thick}}

\newcommand{\Ab}{\mathrm{Ab}}

\newcommand{\op}{\mathrm{op}}
\newcommand{\perf}{\mathrm{perf}}

\newcommand{\iso}{\xrightarrow{\raisebox{-.4ex}[0ex][0ex]{$\scriptstyle{\sim}$}}}

\newcommand{\lto}{\longrightarrow}

\newcommand*{\intref}[2]{\def\tmp{#1}\ifx\tmp\empty\hyperref[#2]{\ref*{#2}}\else\hyperref[#2]{#1~\ref*{#2}}\fi}

\def\A{\mathcal A} 
 
\def\C{\mathcal C}

\def\P{\mathcal P}
\def\Q{\mathcal Q}
 
\def\T{\mathcal T}

\def\X{\mathcal X}

\def\bfD{\mathbf D}

\def\bbZ{\mathbb Z}

\def\p{\phi}

\def\Si{\Sigma}

\begin{document}



\begin{abstract}
The finitistic dimension of a triangulated category is introduced. For
the category of perfect complexes over a ring it is shown that this
dimension is finite if and only if the small finitistic dimension of
the ring is finite.
\end{abstract}

\date{\today}

\maketitle

\section{Introduction}

The finitistic dimension of a ring is a homological invariant which is
conjectured to be finite for any finite dimensional algebra over a
field \cite{Ba1960}. Despite many efforts the conjecture remains
open. So it seems appropriate to take a step back in order to
understand this homological invariant beyond the specific context of
modules over a finite dimensional algebra. This is precisely a
situation where triangulated categories turn out to be useful, because
they enable the transfer of concepts.

Let us consider the category of perfect complexes over a noetherian
scheme. A recent result of Neeman characterises the existence of
bounded t-structures \cite{Ne2022}, confirming a conjecture of
Antieau, Gepner, and Heller \cite{AGH2019}. This requires the scheme
to be of finite dimension, and an analysis of the proof reveals that
there is an analogue of this theorem in a non-commutative setting. It
turns out that the condition on the scheme to be of finite dimension
translates into the finiteness of the finitistic dimension. Thus a
unified proof requires a definition of finitistic dimension for
triangulated categories, and this is precisely what this paper
offers. That there is such a concept is not surprising, because it is
well known that the finiteness of the finitistic dimension of a ring
is a derived invariant, even though the integer values of this
dimension in a derived equivalence class may differ \cite{PX2009}. 

The purpose of this note is to describe the finiteness of the
finitistic dimension of a ring as a property of the triangulated
category of perfect complexes. To this end we introduce for any
triangulated category its finitistic dimension. This is in the spirit
of Rouquier's dimension \cite{Ro2008}; for a wide class of rings its
finiteness characterises regularity. Note that we focus on the small
finitistic dimension. The big finitistic dimension is discussed in an
appendix, where it is shown that its finiteness is an invariant of the
category of perfect complexes. This requires that the ring is an Artin
algebra and it is based on recent work of Rickard \cite{Ri2019}.

For another definition of finitistic dimension for triangulated
categories we refer to \cite{BCMPZ2024}; this is much closer to what
is needed for the existence of bounded t-structures.

\section{The finitistic dimension of a triangulated category}

Let $\T$ be a triangulated category with suspension $\Si\colon\T
\iso\T$. From \cite{BV2003,Ro2008} we recall the following
definition. For an object $X$ in $\T$ and $n\ge 0$ set
\[\thick^n(X):=\begin{cases}
  \add\varnothing&n=0,\\
  \add\{\Si^i X\mid i\in\bbZ\}&n=1,\\
  \add\{\cone\p\mid\p\in\Hom(\thick^1(X), \thick^{n-1}(X))&n>1,
\end{cases}\]
where $\add\X$ denotes the smallest full subcategory containing $\X$
that is closed under finite direct sums and direct summands. Also, set
\[\thick(X):=\bigcup_{n\ge
    0}\thick^n(X).\] The triangulated category $\T$ is
\emph{finitely generated}  when $\T=\thick(X)$ for some object $X$, and $\T$ is  \emph{strongly
  finitely generated} if its \emph{dimension}
\[\dim\T:=\inf\{n\ge 0\mid \exists X\in\T\colon\thick^{n+1}(X)=\T\}\]  
is finite.

We continue with some further notation. For a pair of objects $X,Y$ in
$\T$ and $n\ge 0$ we write  $h(X,Y)\le n$  when for any pair $i,j\in\bbZ$
\[\Hom(X,\Si^iY)\neq 0\neq \Hom(X,\Si^jY) \;\;\implies\;\; |i-j|< n.\]
We set \[\hom^n(X):=\{Y\in\T\mid h(X,Y) \le n\}.\]
The \emph{amplitude} of $X$ is
\[\amp(X):=\sup\{|n|\ge 0\mid\Hom(X,\Si^n X)\neq 0\}.\]

We record some elementary properties of $\hom^n(X)$.

\begin{lem}\label{le:hom-basic}
    \pushQED{\qed}
    For $X\in\T$ and $n\ge 0$ we have
\begin{enumerate}
\item $\hom^0(X)=\{Y\in\T\mid \Hom(X,\Si^iY)=0\text{ for all }i\in\bbZ\}$,      
\item $\hom^1(X)=\{Y\in\T\mid \Hom(X,\Si^iY)\neq 0\text{ for at most one }i\in\bbZ\}$,      
\item $\hom^n(X)=\hom^{n}(\Si^i X)$ for each $i\in\bbZ$,   
\item $\hom^n(X)=\hom^{n+|i|}(X\oplus\Si^i X)$ for each $i\in\bbZ$,
\item $\hom^n(X)\subseteq\hom^n(Y)$ for each $Y\in\add X$,
\item $\hom^n(X)\subseteq\hom^n(Y)$ for each exact triangle $Y'\to Y\to
  Y''\to\Si Y'$ with $Y',Y''\in\add X$.
  \end{enumerate}
\end{lem}
\begin{proof}
  We give some hints and fix objects $Y,Z\in\T$.  The assertions (1)--(3) are
  clear. For (4) one uses that
  \[h(X,Z)\le n\;\;\iff\;\; h(X\oplus \Si^i X,Z)\le n+|i|.\] For (5)
  and (6) consider $Y\in \add X$ or an exact triangle $Y'\to Y\to
  Y''\to\Si Y'$ with $Y',Y''\in\add X$. Then one has
  \[\Hom(Y,Z)\neq 0\;\;\implies\;\;\Hom(X,Z)\neq 0\]
  and therefore
  \[h(X,Z)\le n\;\;\implies\;\;h(Y,Z)\le n.\qedhere\]
\end{proof}

\begin{lem}\label{le:hom-thick}
  Let $Y\in\thick(X)$. For each $p\ge 0$ there are inclusions
  \[\thick^p(Y)\subseteq \thick^q(X) \quad\text{and}\quad
    \hom^p(X)\subseteq \hom^q(Y) \quad\text{for}\quad q\gg 0.\]
\end{lem}  
\begin{proof}
  Let  $Y\in\thick^n(X)$. Then  $\thick^p(Y)\subseteq \thick^{p\cdot n}(X)$. 
The second inclusion follows from the previous lemma by an induction on the
number $n$ of steps needed to build $Y$ from $X$.
\end{proof}

The following definition provides a finiteness condition for
triangulated categories which lies in between `finitely generated' and
`strongly finitely generated'.

Recall from \cite{Or2006} that an object $X$ is \emph{homologically
  finite} if it satisfies
$h(X,Y)<\infty$ for each object $Y$.

\begin{defn}
  An object $X$ of a triangulated category $\T$ is a \emph{finitistic
    generator} of $\T$ if $X$ is homologically finite
  and
  \[\hom^p(X)\subseteq\thick^p(X)\quad\text{for all}\quad p\ge 0.\]
  The \emph{finitistic dimension}\footnote{While the term `finitistic
    dimension' goes back to Bass \cite{Ba1960}, the term `finitistic'
    is often attributed to Hilbert's Programme.  Hilbert intended to
    formalise mathematics based on the use of finite methods. This was
    meant as a contribution towards the foundations of mathematics,
    but his syzygy theorem (about invariants, and in modern language
    about modules of finite projective dimension) may well be
    considered as an early instance for this line of thought. In
    fact, Hilbert mentions `the use or the knowledge of a syzygy' as a
    model for the notion of simplicity when he formulated his
    unpublished 24th problem that asks for `criteria of simplicity, or
    proof of the greatest simplicity of certain proofs'.  For
    Hilbert's 24th problem we refer to \cite{MR2019,Th2005}, and for an excellent
    exposition of Hilbert's finitism, see \cite{Ta2013}.} of $\T$ is
\[\findim\T:=\inf\{\amp(X)\mid X \text{ is a finitistic generator of
  }\T\}.\]
As usual, $\findim\T=\infty$ if no finitistic generator exists.
  \end{defn}

\begin{rem}\label{re:bounded}
  (1) If an object $X$ is homologically finite, then
  $\T=\bigcup_{p\ge 0}\hom^p(X)$. Thus a finitistic generator $X$ of
  $\T$ satisfies $\T=\thick(X)$.
  
(2) If $\T$ admits a finitistic generator, then all objects in $\T$ are
homologically finite.

(3) If $\thick^{n+1}(X)=\T$ for some $X\in\T$, then
$\hom^p(Y)\subseteq\thick^{p}(Y)$ for $Y=X\oplus\Si^n X$ and all
$p\ge 0$.  Thus strong finite generation implies that there is
a finitistic generator, provided that all objects are homologically
finite.
\end{rem}

\section{The small finitistic dimension of a ring}

Let $A$ be a ring and let $\P(A)$ denote the class of $A$-modules $M$
that admit a finite resolution
\[0\lto P_n\lto \cdots \lto P_1 \lto P_0\lto M\lto 0\]
with all $P_i$ finitely generated projective. We denote  
\[\findim A:=\sup\{\pdim M\mid M\in\P(A)\}\]
the \emph{small finitistic dimension} of $A$; this is a slight
variation of the original definition in \cite{Ba1960} which takes the
supremum of $\pdim M$ where $M$ runs through all finitely generated
$A$-modules having finite projective dimension. Our definition seems
natural as the modules in $\P(A)$ are precisely the ones which become
compact (or perfect) when viewed as an object in the derived category
of $A$. Note that both definitions coincide when the ring $A$ is noetherian.

We are now in the position to characterise $\findim A<\infty$ in terms
of the category of perfect complexes. Let $\Perf(A)$ denote the
category of perfect complexes over $A$. 

\begin{thm}\label{th:main}
  For a ring $A$ we have
  \[\findim A<\infty\;\;\iff\;\;\findim\Perf(A)<\infty.\]
More precisely,  let $X$ be a finitistic generator of $\Perf(A)$. Then
there are  $p,q\ge 0$ such that
\[\hom^1(A)\subseteq\hom^p(X)\quad\text{and}\quad\thick^1(X)\subseteq\thick^q(A),\]
which implies
\[\findim \Perf(A)\le\findim A<p\cdot q.\]
\end{thm}

We will use the following well
known fact, which is a consequence of the ghost lemma (cf.\
\cite[Lemma~2.4]{KK2006}).

\begin{lem}\label{le:pdim}
\pushQED{\qed}  For $M\in\P(A)$ and $n > 0$ we have 
  \[\pdim M< n\;\;\iff\;\; M\in\thick^{n}(A)\subseteq\perf(A).\qedhere\]
\end{lem}

\begin{proof}[Proof of Theorem~\ref{th:main}]
The crucial observation is that taking the cohomology of a complex
identifies the objects in $\hom^1(A)$  with the modules in $\P(A)$.
  
Suppose first that $d=\findim A<\infty$. We show by induction on $p\ge 0$ that
  $\hom^p(A)\subseteq\thick^{p+d}(A)$. The case $p=0$ is clear and we
  may assume $p>0$. Let $X\in \hom^p(A)$ and
  suppose its cohomology \[H^n(X)=\Hom(A,\Si^nX)\] is concentrated in
  degrees $0,\ldots,p-1$. Then we may assume $X^n=0$ for $n\ge p$. For $p=1$
the complex $X$ identifies with $H^0(X)\in\P(A)$
and then Lemma~\ref{le:pdim} implies $X\in\thick^{d+1}(A)$. For $p>1$
we may write $X$ as an extension of the truncation
\[X'\colon\quad\cdots \lto X^{p-4}\lto X^{p-3}\lto X^{p-2}\lto
  0\lto\cdots\] and a complex concentrated in
degree $p-1$. We have \[X'\in \hom^{p-1}(A)\subseteq\thick^{p-1+d}(A)\] and therefore
$X\in\thick^{p+d}(A)$. Now set $A'=A\oplus\Si^dA$. Then we have 
\[\hom^{p+d}(A')=\hom^p(A)\subseteq\thick^{p+d}(A)=\thick^{p+d}(A')\]
for $p\ge 0$ and $\hom^d(A')=0$, using Lemma~\ref{le:hom-basic}.
Thus $A'$ is a finitistic generator and $\findim \Perf(A)\le\findim A$
follows since $\amp(A')=d$. 

Suppose now that $X$ is a finitistic generator of $\Perf(A)$.  Using
Lemma~\ref{le:hom-thick} we find  $p,q\ge 0$ such that
\[\hom^1(A)\subseteq\hom^p(X)\qquad\text{and}\qquad\thick^1(X)\subseteq
  \thick^q(A).\] 
Then
\[\hom^1(A)\subseteq\hom^p(X)\subseteq \thick^p(X)\subseteq
  \thick^{p\cdot q}(A).\] 
Thus $\findim A<p\cdot q$ by Lemma~\ref{le:pdim}.
\end{proof}

 A pair of rings $A,B$ is by definition \emph{derived equivalent} if
 there is a triangle equivalence $\Perf(A)\iso\Perf(B)$.
 
\begin{rem}
 It is an immediate consequence of Theorem~\ref{th:main} that for any
 pair of derived equivalent rings $A,B$ we have that $\findim
 A<\infty$ if and only if $\findim B<\infty$. For coherent rings this
 result is due to Pan and Xi \cite{PX2009}.
 \end{rem}

\begin{rem}
  There are examples of rings such that $\findim A=\infty$ and
  $\findim A^\op=0$; cf.\ \cite{Kr2022}.  We have
  $\Perf(A^\op)\simeq\Perf(A)^\op$, and therefore finiteness of 
  finitistic dimension is not a symmetric notion for triangulated categories.
\end{rem}

The following remark collects some known facts about the existence of
strong and finitistic generators of $\Perf(A)$. We give references, and
full proofs are provided in Appendix~\ref{se:fingen}.

\begin{rem}
  Let $A$ be a noetherian ring. Strong finite generation of $\Perf(A)$
  holds if and only if the global dimension is finite. Clearly, this
  implies that $A$ is \emph{regular}, so each finitely generated
  $A$-module has finite projective dimension
  \cite[Proposition~7.25]{Ro2008}. In general (and despite such claims
  in the literature), the converse is not true even when $A$ is
  commutative, because there are regular rings having infinite
  finitistic dimension \cite[Appendix, Example~1]{Na1962}, so
  \[ A \text{ regular} \quad\;\not\!\!\!\!\implies \quad\gldim A<\infty.\]
    However
  when $A$ is in addition semilocal, then $A$ is regular if and only
  if $A$ has finite global dimension \cite[Corollary~2]{XF1994}.  This
  provides many examples where $\Perf(A)$ has a finitistic generator
  but no strong generator, because one may take for $A$ any finite
  dimensional algebra over a field that has infinite global dimension
  and finite finitistic dimension (keeping in mind the
  question/conjecture from \cite{Ba1960} that for finite dimensional
  algebras the finitistic dimension is always finite).
\end{rem}

\begin{exm}
  Let $\T$ be an idempotent complete algebraic triangulated
  category. If $\findim\T=0$, then there exists a ring $A$ with
  $\findim A=0$ and a triangle equivalence $\T\iso \Perf(A)$.
\end{exm}
\begin{proof}
  A generator $X$ of $\T$ with $\amp(X)=0$ is a tilting object. Thus
  for $A=\End(X)$ the functor $\RHom(X,-)$ provides an equivalence
  $\T\iso \Perf(A)$; see
  \cite[Proposition~9.1.20]{Kr2023}. Moreover,
  $\hom^1(A)\subseteq\thick^1(A)$ implies $\findim A=0$ by
  Lemma~\ref{le:pdim}.
\end{proof}

\begin{appendix}
\section{Strong finite generation}\label{se:fingen}

In this appendix we record with proofs some facts that are well known,
but only to those knowing them well. Let $\A$ be an abelian category with
a projective generator $P$. Thus any object $X\in\A$ admits an
epimorphism $P^n\to X$ for some positive integer $n$. Let $\P$ denote
the full exact subcategory of projective objects in $\A$. We consider
the derived categories of bounded complexes and the inclusion
$\P\to\A$ induces a triangle equivalence
\[\bfD^{\mathrm b}(\P)\iso\thick(P)\subseteq\bfD^{\mathrm b}(\A)\]
which we view as an identification.

\begin{prop}
  The following holds for $\bfD^{\mathrm b}(\P)\subseteq\bfD^{\mathrm b}(\A)$.
  \begin{enumerate}
  \item $\bfD^{\mathrm b}(\P)=\bfD^{\mathrm b}(\A)$ $\iff$  $\{X\in\A\mid\pdim X<\infty\}=\A$.
  \item $\dim\bfD^{\mathrm b}(\P)<\infty$ $\iff$ $\gldim\A<\infty$.
\end{enumerate}
\end{prop}

\begin{proof}
  (1) The equality $\bfD^{\mathrm b}(\P)=\bfD^{\mathrm b}(\A)$ holds if and only if each
  object from $\A$, viewed as a complex concentrated in degree zero,
  belongs to $\bfD^{\mathrm b}(\P)$. But an object from $\A$  belongs to $\bfD^{\mathrm b}(\P)$  precisely
  when it has finite projective dimension.
  
  (2) Suppose first that $\gldim\A<n$. Then $\thick^n(P)=\bfD^{\mathrm b}(\A)$
  by \cite[Theorem~8.3]{Ch1998} or \cite[Lemma~2.5]{KK2006}. For the
  converse suppose that $\thick(P)=\thick^n(P)$ for some integer
  $n$. We claim that $\gldim \A<n$. This follows from
  Lemma~\ref{le:pdim} once we know that each object  $X\in\A$ has finite
  projective dimension. To see this choose a projective resolution
  $Q\to X$  and for $1\le i\le n$ let $\p_i$ denote
  the following chain map between appropriate truncations of $Q$:
 \[
    \begin{tikzcd}[column sep=small]
      \cdots\arrow{r}&0
      \arrow{r}\arrow{d}&0\arrow{r}\arrow{d}&Q^{-n-i}\arrow{r}\arrow{d}{\id}&\cdots\arrow{r}&
      Q^{-i}\arrow{r}\arrow{d}{\id}&  Q^{-i+1}\arrow{r}\arrow{d}&0\arrow{r}\arrow{d}&\cdots\\
     \cdots\arrow{r}&0 \arrow{r}&Q^{-n-i-1}\arrow{r}&Q^{-n-i}\arrow{r}&\cdots\arrow{r}&
      Q^{-i}\arrow{r}&  0\arrow{r}&0\arrow{r}&\cdots
    \end{tikzcd}
\]
It is clear that $H^*(\p_i)=0$ for all $1\le i\le n$. Thus
$\Hom(-,\p_i)$ vanishes on $\thick^1(P)$, and therefore $\Hom(-,\p)$
vanishes on $\thick^n(P)$ for $\p=\p_n\circ\cdots\circ\p_1$. This
follows from the ghost lemma and implies that $\p$ is null-homotopic
since $\thick(P)=\thick^n(P)$; see \cite[Lemma~4.11]{Ro2008} or
\cite[Theorem~1]{Ke1965}. Then it follows that the
kernel of the differential $Q^{-n}\to Q^{-n+1}$ is projective. Thus $X$ has finite
projective dimension.
\end{proof}

One may ask for general criteria on $\A$ such that the conditions in
(1) and (2) are equivalent.  To indicate some connection we include
the following.

\begin{lem}
Suppose that every non-zero object in $\A$ has a
projective cover and a maximal subobject. Then we have
\[\gldim\A<\infty\quad\iff\quad\{X\in\A\mid\idim  X<\infty\}=\A.\]
\end{lem}
\begin{proof}
One direction is clear and requires no assumption on $\A$. For the
converse it is convenient to identify $\A$ with the category of
finitely presented modules over $\End(P)$ via the functor
$X\mapsto\Hom(P,X)$.  The assumption on $\A$ implies that this ring 
is semiperfect, and therefore $S=P/\rad P$ is semisimple. We claim
that $X\in\A$ is projective if and only if $\Ext^1(X,S)=0$. To see
this choose a projective cover
\[\xi\colon\quad 0\lto\Omega(X)\lto P(X)\lto X\lto 0.\]
If $\Omega(X)\neq 0$, then there is a non-zero morphism
$\pi\colon \Omega(X)\to S$ such that $\pi\circ\xi\neq 0$ in
$\Ext^1(X,S)$ since $\Omega(X)\subseteq \rad P(X)$. From this
characterisation of projectivity it follows that $\gldim\A$ is bounded
by $\idim S$.
\end{proof}

\section{The big finitistic dimension}

In this appendix we show for an Artin algebra that the finiteness of
the big finitistic dimension is an invariant of the category of
perfect complexes, using some recent work of Rickard
\cite{Ri2019}. The argument is different from the one for the small
finitistic dimension, and also much less direct. We refer to
\cite{BCMPZ2024} for the notion of a big finitistic dimension for
compactly generated triangulated categories.

Let $\C$ be a category with an equivalence $\Si\colon\C\iso\C$. For a
class of objects $\X$ in $\C$ set
\[\X^+:=\{Y\in\C\mid \Hom(X,\Si^n Y)=0\text{ for all }X\in\X, n\ll 0\}\]
and analogously
\[\X^-:=\{Y\in\C\mid \Hom(X,\Si^n Y)=0\text{ for all }X\in\X, n\gg 0\}.\]
Also set $\X^{\mathrm b}:=\X^-\cap\X^+$ and
\[\X^\perp:=\{Y\in\C\mid \Hom(X,\Si^n Y)=0\text{ for all }X\in\X, n\in\bbZ\}.\]

Let $\T$ be an essentially small triangulated category with suspension
$\Si\colon\T\iso\T$. We write $\Mod\T$
for the category of $\T$-modules, i.e.\ additive functors
$\T^\op\to\Ab$. The suspension extends to an equivalence
$\Mod\T\iso\Mod\T$ via the assignment $X\mapsto X\circ \Si^-$. 
The full subcategory of injective objects in
$\Mod\T$ is denoted by $\Inj\T$.  

Let $A$ be an Artin algebra. We write $\Mod A$ for the category of
$A$-modules and consider its derived category. The category of perfect
complexes is viewed as a subcategory and we set
\[\P:=\Perf(A)\subseteq\bfD(\Mod A).\]
The Yoneda embedding $\P\to\Mod\P$ identifies $\P$ with a full
subcategory $\Q\subseteq \Mod\P$; it consists of injective modules by
Lemmas~\ref{le:res-yoneda} and \ref{le:pinj} below.

The \emph{big finitistic dimension} of $A$ is by definition
\[\Findim A:=\sup\{\pdim M\mid M\in\Mod A,\,\pdim M<\infty\}.\]

 \begin{thm}\label{th:app}
For an Artin algebra $A$ the following conditions are equivalent.   
\begin{enumerate}
\item The big finitistic dimension of $A$ is finite. 
\item In the derived category $\bfD(\Mod A)$ we have $((\P^{\mathrm b})^{\mathrm b})^\perp\cap\P^+=\{0\}$. 
\item In the category of injective $\Perf(A)$-modules we have $((\Q^{\mathrm b})^{\mathrm b})^\perp\cap\Q^+=\{0\}$. 
 \end{enumerate}
\end{thm}

The proof requires some preparations. Let $k$ be a commutative ring
and $A$ a $k$-algebra. We denote by $D=\Hom_k(-,E)$ the Matlis duality
given by a minimal injective cogenerator $E$ over $k$. The functor $D$
induces dualities
\[\Mod A\lto\Mod A^\op \qquad\text{and}\qquad \bfD(\Mod
  A)\lto\bfD(\Mod A^\op)\] such that for all $X$ in $\bfD(\Mod A)$ and
$Y$ in $\bfD(\Mod A^\op)$ there is an isomorphism
\[
  \Hom_{\bfD(A)}(X,DY)\cong\Hom_{\bfD(A^\op)}(Y,DX). \label{eq:dual}\tag{$\ast$}
\]

For any compactly generated triangulated category there is a notion of
\emph{purity} and in particular the notion of a \emph{pure-injective
  object}; see \cite[\S1.4]{Kr2000}. The following is basically all we
need to know.

\begin{lem}\label{le:res-yoneda}
  The restricted Yoneda functor
  \[\bfD(\Mod A)\lto \Mod\P,\quad X\mapsto \bar X:=\Hom_{\bfD(A)}(-,X)|_\P\] identifies
  the full subcategory of pure-injective objects of $\bfD(\Mod A)$
  with the category of injective objects of $\Mod\P$. 
\end{lem}
\begin{proof}
  See \cite[Corollary~1.9]{Kr2000}.
\end{proof}

\begin{lem}\label{le:pinj}
  Any complex of the form $DX$ is pure-injective.
\end{lem}
\begin{proof}
  We view $A$ as a differential graded algebra and have an isomorphism
  of functors
  \[\Hom_k(-\otimes_A X,E)\cong\Hom_A(-,\Hom_k(X,E))\]
  on the category of dg $A$-modules, which take pure exact sequences
  to exact sequences of abelian groups. Thus $DX$ is a pure-injective object in the
  category of dg $A$-modules, and therefore the summation morphism
  $\coprod_I DX\to DX$ factors through the canonical morphism
  $\coprod_I DX\to\prod_I DX$ for every set $I$. It follows that $DX$
  is a pure-injective object in $\bfD(\Mod A)$; see \cite[Theorem~1.8]{Kr2000}.
\end{proof}

Observe that in $\bfD(\Mod A)$ the objects in $\P^{\mathrm b}$ are the complexes
with bounded cohomology, whereas the objects in $(\P^{\mathrm b})^{\mathrm b}$ are (up to
isomorphism) the bounded complexes of injective $A$-modules. Thus
\[\P^{\mathrm b}=\bfD^{\mathrm b}(\Mod A)\qquad\text{and}\qquad (\P^{\mathrm b})^{\mathrm b}=\bfD^{\mathrm b}(\Inj
  A).\] We set $S=A/\rad A$ and note that $(\P^{\mathrm b})^{\mathrm b}= S^{\mathrm b}$ since any
$A$-module $M$ admits a finite filtration
$0=M_0\subseteq M_1\subseteq\cdots\subseteq M_n=M$ such that each
quotient $M_{i}/M_{i-1}$ is semisimple, so a direct summand of a
direct sum of copies of $S$. Similarly, we have $((\P^{\mathrm b})^{\mathrm b})^\perp= (DA)^\perp$.

\begin{lem}\label{le:Yoneda}
  The restricted Yoneda functor $\bfD(\Mod A)\to \Mod\P$ identifies
\begin{enumerate}
\item the pure-injective objects in $\P^+$ with the objects in
  $\Q^+\subseteq\Inj\P$,
  \item the pure-injective objects 
    in $(\P^{\mathrm b})^{\mathrm b}$ with the objects in $(\Q^{\mathrm b})^{\mathrm b}\subseteq\Inj\P$, 
  \item the pure-injective objects 
    in $((\P^{\mathrm b})^{\mathrm b})^\perp$ with the objects in $((\Q^{\mathrm b})^{\mathrm b})^\perp\subseteq\Inj\P$.
\end{enumerate}    
\end{lem}
\begin{proof}
  The restricted Yoneda functor $\bfD(\Mod A)\to \Mod\P$ identifies
  the full subcategory of pure-injective objects of $\bfD(\Mod A)$
  with the category of injective objects of $\Mod\P$ by
  Lemma~\ref{le:res-yoneda}. Now the first assertion follows
  since $\P\iso\Q\subseteq\Inj\P$. For the second assertion observe
  that $(\Q^{\mathrm b})^{\mathrm b}= \bar S^{\mathrm b}$, keeping in mind that $S$ is
  pure-injective (as an $A$-module, and therefore also as an object of
  $\bfD(\Mod A)$). Then the pure-injectives in $S^{\mathrm b}\subseteq \bfD(\Mod
  A)$ identify with the objects in $\bar S^{\mathrm b}\subseteq\Inj\P$.  For the
  third assertion one uses that $((\Q^{\mathrm b})^{\mathrm b})^\perp= \bar Q^\perp$
  for $Q=DA$. 
\end{proof}

\begin {proof}[Proof of Theorem~\ref{th:app}] The proof combines
  results from \cite{Ri2019} with Lemma~\ref{le:Yoneda}.
  
(1) $\Leftrightarrow$ (2) This is a reformulation of Theorem~4.4 in
\cite{Ri2019}.

(2) $\Rightarrow$ (3) This is clear since  $((\Q^{\mathrm b})^{\mathrm b})^\perp$ and $\Q^+$ in
$\Inj\P$ identify with full subcategories  of $((\P^b)^{\mathrm b})^\perp$ and $\P^+$ in
$\bfD(\Mod A)$ by Lemma~\ref{le:Yoneda}.

(3) $\Rightarrow$ (1) Suppose that the big finitistic dimension of $A$
is infinite. The proof of Proposition~5.2 in \cite{Ri2019} yields a
bounded above complex $X\neq 0$ of projective $A^\op$-modules satisfying
$\Hom_{\bfD(A^\op)}(X,\Si^n A)=0$ for all $n\in\bbZ$, and therefore
$\Hom_{\bfD(A)}(\Si^n DA,DX)=0$ for all $n\in\bbZ$, using the duality isomorphism \eqref{eq:dual}. Thus 
$\Hom_{\bfD(A)}(Q,DX)=0$ for all $Q\in(\P^{\mathrm b})^{\mathrm b}$, while $0\neq
DX\in\P^+$. It remains to note that $DX$ is pure-injective by Lemma~\ref{le:pinj}.
\end{proof}

\end{appendix}

\subsection*{Acknowledgement}

I am grateful to Janina Letz for useful comments and for sharing her
computations of the finitistic dimension in some interesting
examples. Additional thanks to Hongxing Chen for pointing out some
inaccuracies in a previous version.
This work was supported by the Deutsche
Forschungsgemeinschaft (SFB-TRR 358/1 2023 - 491392403).

\end{document}